\documentclass{amsart}
\usepackage{amssymb}
\usepackage{hyperref,url}
\usepackage{graphicx} 
\usepackage{caption}
\usepackage{cleveref}
\usepackage{tikz} 
\usetikzlibrary{cd}
\usepackage[utf8]{inputenc} 
\usepackage[T1]{fontenc}
\usepackage{array}
\usepackage{enumitem}

\usepackage[maxbibnames=99]{biblatex}
\def\wid{\operatorname{wd}}

\usepackage{multicol,multirow}
\usepackage{color}
\definecolor{red}{rgb}{1.00,0.00,0.00}

\newtheorem{theorem}{Theorem}[section]
\newtheorem{lemma}[theorem]{Lemma}
\newtheorem{corollary}[theorem]{Corollary}
\newtheorem{proposition}[theorem]{Proposition}

\newtheorem{example}[theorem]{Example} 
\newtheorem{remark}[theorem]{Remark}
\newtheorem{definition}[theorem]{Definition}

\newtheorem{question}{Question}
\newtheorem{notation}{Notation}

\newtheorem*{acknowledgment*}{Acknowledgment}

\def\N{\mathbb{N}}

\newcommand{\bs}{\setminus}
\newcommand{\q}{\quad}
\newcommand{\la}{\langle}
\newcommand{\ra}{\rangle}

\newcommand*{\SM}[1]{S^e(#1)}
\newcommand*{\IM}[1]{I^e(#1)}
\newcommand*{\RM}[1]{R^e(#1)}

\renewcommand*{\S}[2]{S_{#1}(#2)}
\newcommand*{\Se}[2]{S^e_{#1}(#2)}
\renewcommand*{\Re}[2]{R^e_{#1}(#2)}

\newcommand*{\Ie}[2]{I^e_{#1}(#2)}

\begin{document}

\title{Structure and Symmetry of Sally type Semigroup Rings}

\author{Srishti Singh}
\address{University of Missouri, USA}
\email{spkdq@umsystem.edu}
\author{Hema Srinivasan}
\address{University of Missouri, USA}
\email{srinivasanh@missouri.edu}
 
\begin{abstract}
Consider a numerical semigroup minimally generated by a subset of the interval $[e,2e-1]$ with multiplicity $e$ and width $e-1$. Such numerical semigroups are called Sally type semigroups. We show that the defining ideals of these semigroup rings, when the embedding dimension is $e-2$, generically have the structure of the sum of two determinantal ideals. More generally, Sally type numerical semigroups with multiplicity $e$ and embedding dimension $d=e-k$ are obtained by introducing $k$ gaps in the interval $[e,2e-1]$.  It is known that for $k =2$, there is precisely one such semigroup that is Gorenstein, and it happens when one deletes consecutive integers. Let $\Se{k}{j}$ denote the Sally type numerical semigroup of multiplcity $e$, embedding dimension $e-k$ obtained by deleting the $k$ consecutive integers $j, j+1, \ldots, j+k-1$.We prove that for any $1\le k < e/2$, the semigroup $\Se{k}{j}$ is Gorenstein if and only if $j=k$. We construct an explicit minimal free resolution of the semigroup ring of $\Se{k}{k}$ and compute the Betti numbers.  In general, we characterize when $\Se{k}{j}$ are symmetric and construct minimal resolutions for these Gorenstein semigroups rings.  
\end{abstract}

\maketitle

\section{Introduction}
Numerical semigroups of multiplicity $e$ and width $e-1$ are called \textit{Sally type semigroups}. These semigroups are carved out of the interval semigroups $\la e,e+1,...,2e-1 \ra$ by introducing gaps between $e$ and $e-1$.  The notion of Sally semigroups was coined by Herzog and Stamate for the symmetric Sally type semigroups. The study of such semigroups is motivated by a remarkable result of J. Sally in which she introduces a specific numerical semigroup in order to construct a Gorenstein ideal of multiplicity two more than the embedding dimension whose associated graded ring is Cohen–Macaulay (\cite{Sally}). In particular, Sally showed that when a one-dimensional Gorenstein ring has multiplicity $e$ and embedding dimension $e-2$, its defining ideal corresponds to a symmetric numerical semigroup; Herzog and Stamate called these semigroups Sally and showed that $S=\la e,e+1,e+4,...,2e-1 \ra$, Sally's original construction, is indeed such a semigroup (\cite{HerzogStamate}). 

The goal of this paper is to understand the structure of Sally type semigroups and determine when they are symmetric.  The interval semigroups have been well-studied (\cite{Rosales}, \cite{HerzogStamate}). Sally type semigroups are precisely those generated by a subset of an interval $[e,2e-1]$ of maximal embedding dimension. What happens when we take out one generator from this interval? Two generators? $k$ generators? The beginning of the story is in Sally's paper (\cite{Sally}), where her example sheds light on the $k=2$ case. In \cite{DGS3}, it was shown that Sally’s semigroup is the only symmetric Sally type semigroup obtained by deleting two generators, and a minimal generating set for its defining ideal was computed. Continuing the story in \cite{DGS2}, the authors study the cases $k=1$ and $k=2$, and construct minimal free resolutions for the symmetric semigroups here. In \cite{Bhardwaj25}, the projective monomial curves associated to the $k=2$ Sally type semigroups are investigated, and characterizations of when these curves are Cohen–Macaulay or Gorenstein are obtained.

In this paper, we begin with the interval semigroup $\la e,e+1,...,2e-1 \ra$ and remove $k$ consecutive generators. We describe a determinantal structure for the resulting Sally type semigroup rings, identify which among them are symmetric, and construct their minimal free resolutions. The embedding dimension $\nu$ of these semigroup rings always satisfies $\nu \le e-1$. In the extremal case $\nu =e-1$, i.e., for the Sally type semigroups $$\la e,e+1,...,2e-2\ra \q \text{or} \q \la e+1,e+2,...,2e-1\ra$$ having maximum possible embedding dimension $e-1$, the generators form an arithmetic sequence, and the semigroup ring is defined by a binomial determinantal ideal. When $\nu =2$, we show that the defining ideal can be expressed as the ideal of $2 \times 2$ minors of two matrices. We settle the symmetric classification for all $k \leq \frac{e}{2}$, proving that for each such $k$, there is exactly one Sally type semigroup that is symmetric, and we compute minimal free resolutions in these cases. If $k> \frac{e}{2}$, we characterize when these are symmetric and prove that there is atmost one Sally type semigroup that is symmetric and when there is one, the minimal resolutions can be found in the same way as $j=k \le \frac{e}{2}$ and surprisingly, the Betti numbers are exactly the same. 

\vspace{1cm}

\section{Determinantal Structure for Sally Type Semigroups of Embedding Dimension $e-2$}

Let $[t,s]$ denote the set of non-negative integers $i$ with $t\le i\le s$. A numerical semigroup $S$ is a submonoid of $\N$ such that $S \bs \N$ is finite. Write $S= \la s_1,...,s_t \ra$ to denote that $S$ is minimally generated by $s_1,...,s_t \in \N$ and so $\gcd(s_1,...,s_t)=1$. The difference $\wid(S)=s_t-s_1$ is the \textit{width} of $S$, while $s_1$ is its \textit{multiplicity} and $t$ is its \textit{embedding dimension}.

\begin{definition}
    A numerical semigroup is a \textit{Sally Type Semigroup} if it has multiplicity $e$ and width $e-1$.
\end{definition}

\begin{remark}
    A Sally type semigroup has the minimal generating set $$[e,2e-1] \bs \{e+j_1,e+j_2,...,e+j_k\}, \q 1 < j_1 <...< j_k < e-1.$$
\end{remark}

In this section, we give a determinantal structure for the ideals defining these semigroup rings when $k=2$. A complete description of a minimal generating set for these ideals is found in \cite{DGS3}, where the minimal number of generators for these ideals is computed using Hochster's formula, and the minimal generators have been explicitly described. Here, we will use the following beautiful criterion by W. Gastinger \cite{Gastinger} on numerical semigroups. We state a simplified form that will suffice for our purposes with a hint of proof. 

\begin{theorem}[\cite{Gastinger}]\label{thm:J=I}
    If $S = \la a_1, \ldots, a_t \ra$ is a numerical semigroup and $I \subset k[x_1, \ldots, x_t]$ is its defining ideal, then any other binomial ideal $J \subset I$  is, in fact, equal to $I$ if the $k$-vector space dimension of $k[x_1, \ldots, x_t]/(J, x_1)$ is $a_1$.
\end{theorem}
\begin{proof}
Set $R:=k[x_1,...,x_t]$. Since $J+(x_1) \subseteq I+(x_1)$, we get a surjective $R$-module homomorphism $$\frac{R}{(J,x_1)} \to \frac{R}{(I,x_1)}.$$ In fact, this is a surjective map between $k$-vector spaces as any $R$-module is a $k$-module. Since $$\dim_k \left( \frac{R}{(J,x_1)} \right) = a_1 = \dim_k \left( \frac{R}{(I,x_1)} \right),$$
this map is an isomorphism. Thus, $(I,x_1)=(J,x_1)$.  Now, if $I\neq J$, let $a-b $ be a binomial of smallest weighted degree in $I$ that is not in $J$.  If $x_1$ divides $a$, then for some $a' \in R$, $x_1a'-b \in I \implies b\in I+(x_1) = J+(x_1)$.  Hence $b = c+x_1d $ for some $c\in J$.  Since $J$ is homogeneous, we can take $c$ and $x_1d$ to be of the same degree as $b$.    So, $b-x_1d \in J \subset I$.  Hence $x_1(a'-d)\in I$.  But $I$ is prime. So $a'-d\in I$, contradicting the minimality of $a-b$.  So, $x_1$ does not divide $a$ or $b$. This implies $a-b = c+x_1d$ where $c\in J$. So $a-b-x_1d \in J \subset I$, which further implies $d\in I$ and is of smaller degree. It now follows that $d\in J$. So, $a-b\in J$ , a contradiction.  Thus, $I = J$.

\end{proof}

\begin{notation}\label{not:A}
Let $\SM{m,n}$ denote the Sally-type numerical semigroup of embedding dimension $e-2$, width $e-1$, and multiplicity $e$.  To be precise, $\SM{m,n} = \la e, \ldots, e+m-1, e+m+1, \ldots, e+n-1, e+n+1, \ldots, 2e-1 \ra. $ 

Let $k$ be a field, and $\RM{m,n}:= k[X_i, i \in [0,e-1]\bs\{m,n\}]$ a polynomial ring in $e-2$ variables. The semigroup ring of $\SM{m,n}$ is $$k[\SM{m,n}] \cong \displaystyle{\frac{\RM{m,n}}{\IM{m,n}}},$$ where $\IM{m,n}$, called the defining ideal of $\SM{m,n}$, is the kernel of the $k$-algebra homomorphism $\RM{m,n} \to k[t]$ defined by $X_i \mapsto t^{e+i}$ for $i \in [0,e-1]\bs\{m,n\}$. The ring $k[\SM{m,n}]$ is weighted homogenous with the grading $X_i = e+i$.  For $0\le t\le 2e-2$, let $X_{e+2t} = X_t^2, X_{e+2t-1} = X_tX_{t-1}$.  Thus, the degree of $X_i = e+i, 0\le i\le 3e-2$.  Let  $I_2(M)$ denote the ideal of $2 \times 2$ minors of a matrix $M$.

\end{notation}

\begin{theorem}\label{thm:remove2determ}
    $\IM{m,n}= I_2(A^e(m,n))+I_2(B^e(m,n))$, where

\begin{align*}
A^e(m,n) &= \begin{bmatrix}
        X_0 & \cdots & X_{m-2} & X_{m+1} & \cdots & X_{n-2} & X_{n+1} & \cdots & X_{e-1} \\
        X_1 & \cdots & X_{m-1} & X_{m+2} & \cdots & X_{n-1}& X_{n+2} & \cdots & X_0^2
        \end{bmatrix} \\
B^e(m,n) &= \left[\begin{array}{cccccccccccccc}
          X_0 & X_1 & \cdots & X_{m-3} & {\bf {X_{m-1}}} & X_{m+1} & \cdots & X_{n-3} & {\bf {X_{n-1}}} & X_{n+1} & \cdots & X_{e-2} & X_{e-1} & X_{0}^2\\ 
          X_2 & X_{3} & \cdots & X_{m-1} & {\bf X_{m+1}} & X_{m+3} & \cdots & X_{n-1} & {\bf X_{n+1}} & X_{n+3} & \cdots & X_0^2 & X_0X_1 & X_1^2  
    \end{array}\right], \\ 
    &  \q \q  m \neq 2, (m,n) \neq (2,3) \\
B^e(2,n) &= \begin{bmatrix}
         {\bf {X_1}}  & X_3 & \cdots & X_{n-3} & {\bf {X_{n-1}}} & X_{n+1} & \cdots &  X_{e-2} & X_{e-1} & X_{0}^2\\
         {\bf { X_3 }} & X_5 & \cdots & X_{n-1} & {\bf X_{n+1}} & X_{n+3} & \cdots & X_0^2 & X_0X_1 & X_1^2 \\ 
    \end{bmatrix} +(X_0X_{n+1}-X_3X_{n-2}), n\ge 6 \\
B^e(2,3) &= \begin{bmatrix}
        {\bf {X_1}} & {X_4} & \cdots & X_{e-3} &  X_{e-2} & X_{e-1}\\
        {\bf{X_4}} & {X_7} & \cdots & X_0^2 & X_0X_1 & X_1^2 \\ 
    \end{bmatrix}  \\
B^e(2,4) &= \begin{bmatrix}
        {\bf {X_1}} & {\bf {X_3}} &  \cdots &  X_{e-2} & X_{e-1} & X_{0}^2\\
        {\bf{X_3}} & {\bf{X_5}} & \cdots & X_0^2 & X_0X_1 & X_1^2 \\ 
    \end{bmatrix} \\
B^e(2,5) &= \begin{bmatrix}
        {\bf {X_0}} & X_1 & {\bf {X_3}} &  \cdots & X_{e-3} & X_{e-2} & X_{e-1} \\
        {\bf{X_3}} & X_4 & {\bf{X_6}} & \cdots & X_0^2 & X_0X_1 & X_1^2 \\ 
    \end{bmatrix} \\
B^e(3,4) &= \begin{bmatrix}
        {\bf {X_2}} & X_5 & \cdots & X_{e-3} & X_{e-2} & X_{e-1} & X_{0}^2 & X_0X_1 \\
        {\bf{X_5}} & X_8 & \cdots & X_0^2 & X_0X_1 & X_1^2 & X_1X_2 & X_2^2 \\ 
    \end{bmatrix}
\end{align*}
\end{theorem}

\begin{proof}
Let $I_2(A^e(m,n))+I_2(B^e(m,n))= J$. Since the difference of the weighted degree of the entries in each column is the same in all the columns of the matrix $A^e(m,n)$, $I_2(A^e(m,n)) \subset \IM{m,n}$.  Similarly, the difference in the weighted degree of the entries in  every column  in $B^e(m,n)$ is the same.  Hence, the binomial ideal $J\subset \IM{m,n}$.

We will show that $R':=\RM{m,n}/(J,X_0)$ is a $k$-vector space of dimension $e$. \Cref{thm:J=I} will then imply that $J= \IM{m,n}$. By abuse of notation, denote by $f$ the image of $f \in \RM{m,n}$ in $R'$ under the projection $\RM{m,n} \to R'$. \\

\textbf{Claim:} $\{1, X_i, X_1X_{m-1}, X_1X_{n-1} \mid i \in [1,e-1]\bs\{m,n\}\}$ is a basis for $R'$ over $k$.  Hence, $\dim R' = e$.

We will check the generic case $m \neq 2, (m,n) \neq (2,3)$.  The rest is similar with straightforward calculations. 

Since $X_0\in (J,X_0)$, $X_1X_j \in (J,X_0)$ for $j\neq m-1,n-1$ as can be seen using $I_2(A^e(m,n))$.  Now $X_iX_j \in (J,X_0) , i+j\le m-1$, as, using binomials from $I_2(A^e(m,n))$ and $I_2(B^e(m,n))$, $$X_iX_j =(X_iX_j-X_{i-1}X_{j+1})+\ldots +(X_2X_{i+j-2}-X_1X_{i+j-1})+X_1X_{i+j-1}.$$  Since $X_2X_j \in (J,X_0), j\ge m-1, j\neq n-2$, using $I_2(B^e(m,n))$, and $X_i X_j = X_2X_{i+j-2} \text { mod } I_2(A^e(m,n))+I_2(B^e(m,n))$, we get $X_iX_j \in (J,X_0), i+j \neq m, n$. 
Finally, $X_iX_j = X_1X_{m-1} \text { mod }(J,X_0), i+j = m$ and $X_iX_j = X_1X_{n-1} \text { mod }(J,X_0), i+j = n$ as can be seen from $I_2(A^e(m,n))$.

Thus, the dimension of $R'$ over $k$ is at most $e$.  However, it is at least $e = \dim_k \left(\RM{m,n} / \IM{m,n},X_0 \right)$.  Hence, we have $\dim_k \RM{m,n}/(J,X_0) = \dim_k \RM{m,n}/(\IM{m,n},X_0)$ and the equality follows from Gastinger's criterion (\Cref{thm:J=I}). 

As a separate check, we show that these set of $e$ elements in the claim are indeed $k$-linearly independent. 

For any $f \in \{1, X_i, i\in[1,e-1]\bs \{m,n\}\}$, $$f \in (J,X_0) \Leftrightarrow f-X_0g \in J \text{ for some } g \in \RM{m,n},$$ which is not possible as $f-X_0g$ cannot be homogeneous. If $X_1X_{m-1} \in (J,X_0)$, then for some $g \in \RM{m,n}$, $$X_1X_{m-1} - X_0g \in J \Leftrightarrow X_0g \text{ has weight } 2e+m \Leftrightarrow g \text{ has weight } m,$$ a contradiction. Similarly, $X_1X_{n-1} \not\in (J,X_0)$. To show that the images of all these elements are linearly-independent over $k$ in $R'$, we will prove no $k$-linear combination of $$1, X_i, i\in[1,e-1]\bs \{m,n\}, X_1X_{m-1}, X_1X_{n-1}$$ can be in $(J,X_0)$. But this is true as for any such $k$-linear combination $f$, $f-X_0g$ is homogeneous for some $g \in \RM{m,n}$ $ \Leftrightarrow g \text{ has weight } i \in [1,e-1]\bs\{m,n\} \cup \{e+m,e+n\},$ which is impossible. 

This completes the proof.

\end{proof}

It is known that $\SM{m,n}$ is symmetric if and only if $(m,n)=(2,3)$ (\cite{DGS3}). We now want to study this in more generality, i.e.,
\begin{question}\label{ques:whensymmetric}
    When are Sally type semigroups symmetric?
\end{question}

\begin{notation}\label{not:B} Let $\Se{k}{j}$ denote the numerical semigroup of multiplicity $e$ and width $e-1$ which misses $k$ consecutive elements $j,j+1, \ldots, j+k-1$ from the set $[e,2e-1]$, i.e. $$\Se{k}{j} = \langle e,\ldots e+j-1,e+j+k \ldots 2e-1 \rangle.$$ 
\end{notation}
Thus, $\Se{k}{j}$ has embedding dimension $e-k$. When it is clear from context, we skip the superscript $e$.\\
For a numerical semigroup, let $G$ denote its set of gaps, $g$ the genus, i.e., $g=|G|$, and $F$ the Frobenius number, the largest gap. Thus,  a numerical semigroup is symmetric {\textit if and only if}
\begin{equation}\label{eq:2g=f+1}
     2g = F+1.
\end{equation}
We concentrate on embedding dimension $4$ and higher as the symmetry in embedding dimension $3$ or less is clear. Thus, from now on, assume $e-k \geq 4$.

The remainder of this paper is devoted to answering \Cref{ques:whensymmetric} for these Sally type semigroups, and constructing minimal free resolutions for the associated Gorenstein semigroup rings. We do this in two cases: $k \leq e/2$ and $k>e/2$.

\vspace{1cm}

\section{Symmetric Sally Type Semigroups When $k \leq e/2$.}

The following three results characterize when $\S{k}{j}$ is symmetric. 

\begin{lemma}\label{lem:smallgaps}
    $\displaystyle{G \supset [1,e-1] \cup [e+j,e+j+k-1], \q \text{and} \q  [2e,2e+2j-2] \subseteq \S{k}{j}.}$
\end{lemma}
\begin{proof}
    The latter is true as any element in $[2e,2e+j-1]$ can be expressed as $(e)+ (e+x)$ for some $0 \leq x \leq j-1,$ and anything in $[2e+j,2e+2j-2]$ looks like $(e+j-1) + (e+y+1)$ for some $0 \leq y \leq j-2$, all of which lie in $\S{k}{j}$.
\end{proof}

\begin{theorem}\label{thm:sym<e/2}
    Suppose $k \le \frac{e}{2}$, $(j,k)\neq (1,\frac{e}{2})$.  Then $\S{k}{j}$ is symmetric if and only if $j=k$. For $(j,k)= \left(1, \frac{e}{2}\right)$, $\S{k}{j}$ is always symmetric. 
\end{theorem}

\begin{proof}
    We want to utilize \Cref{eq:2g=f+1} and so we will compute the genus and Frobenius number of these semigroups.

\noindent 
\underline{Case I}. $j \geq k+1$. In this case, $2e+2j-2-(e+j+k-1) \geq e$ so that \Cref{lem:smallgaps} gives us $e$ consecutive elements in $\S{k}{j}$ after $e+j+k-1$. So $F=e+j+k-1$ and \Cref{lem:smallgaps} implies $2g=2(e+k-1)= e+k-1+e+k-1 >e+k+j=F+1$. Thus, $\S{k}{j}$ is not symmetric.

\noindent
\underline{Case II}. $j \leq k$. We will show that in this case, $F=2e+k+j-1$ and $g=e+2k-j$. First we claim $[2e+2j-1,2e+j+k-1] \subseteq G$. To see this, suppose that for some $-1 \leq u \leq k-j-1$, $2e+2j+u \in \S{k}{j}$. We may write $2e+2j+u = e+a + e+b \Rightarrow 2j+u = a+b$ for some $a \in \{0,...,j-1,j+k,...,e-1\}$ and some $b\geq 0$. If $a \in [0,j-1],$ then $a+b=2j+u \implies b >j+u \implies b \geq j+k$. Then necessarily $a<u-k+j <0$, a contradiction. If $a \in [j+k,e-1],$  then $2j+u \leq  j+k-1 < j+k \leq a+b$, also a contradiction. Thus, 
\begin{equation*}\label{eq:gapsjleqk}
    G \supseteq [1,e-1] \cup [e+j,e+j+k-1] \cup [2e+2j-1,2e+j+k-1].
\end{equation*}
Hence $g \geq e+k-1+k-j+1= e+2k-j$.  
Next, 
\begin{equation}\label{eq:nongapsjleqk}
    [e+j+k,2e+2j-2] \subseteq \S{k}{j} \implies [2e+j+k,3e+3j-3] \subseteq \S{k}{j}
\end{equation}

When $2j \geq k+2$, $3e+3j-3-(2e+j+k)+1 \geq e$ so that \Cref{eq:nongapsjleqk} gives us $e$ consecutive elements in $\S{k}{j}$. When $2j=k+1$, this only gives us $e-1$ consecutive elements in $\S{k}{j}$. However, $3e+3j-2 = 3e+j+k-1 = (e+j+k)+(2e-1) \in \S{k}{j}$. Thus, for $2j\ge k+1$, we have $F= 2e+j+k-1$ and $g = e+2k-j$.  

Further, $2e+2j+2k \le 3e+3j-2 $ whenever $2k\le e+j-2$.  In this case, we will have no gaps between $3e+3j-3$ and $2e+2j+2k$.  However, since $[e+j+k, 2e+2j-2] \subset \S{k}{j}$, we must have $[2e+2j+2k, 4e+4j-4]\subset \S{k}{j}$.  But then $4e+4j-4-(2e+j+k-1) \ge e$ and hence we have $F = 2e+j+k-1$ and $g= e+2k-j$ again.  Thus, the only remaining case is if $2k > e+j-2$ which can happen only if $e=2k$ and $j =1$.  

Now, if $e\neq 2k, j\neq 1$, we have $F$ and $g$ as above. Symmetry occurs only if $2e+j+k = 2e+2k-2j $ which implies $j=k$. 

Now for the remaining case $e=2k, j=1$:  Here, $4e+1 \notin \S{k}{j}$.  However, $4e+2 = 2(e+k+1) \in \S{k}{j}$.  Indeed, $[4e+2, 5e+5j-5]= [2e+2k+2j, 5e+5j-5]$ is in the semigroup.  But $5e+5j-4 = 5e+1 = 2e-1+2e+2k+2j$. Hence, the interval 
$[4e+2,5e+5j-4] $ which is of length $e$ is in the semigroup. So we get $F= 4e+1, \; g= e+2k-j+1+1=e+2k+1= 2e+1$.  This is symmetric if and only if $4e+2= 4e+2$. which is true.   

Thus, $k\le \frac{e}{2}$ implies $\S{k}{j}$ is symmetric if and only if $k=j$ unless $j=1, e=2k$, and $\S{e/2}{1}$ is always symmetric. 
\end{proof}

The proof of the above \Cref{thm:sym<e/2} also shows the following.
\begin{corollary}\label{cor:sym<e/2}
   Suppose $j\ge k+1$ or $2j \ge k+1$ or $2k\le e+j-2$.  Then  $\S{k}{j}$ is symmetric if and only if $k=j$.  
\end{corollary}
\begin{proof}
    The first part of the proof of \ref{thm:sym<e/2} shows the result if $j\ge k+1$ or $2j\ge k+1$ or $2k\le e+j-2$ as in these cases, $F= 2e+j+k-1$ and $g= e+2k-j$.  However, if $j=1$, then again $2k\le e+j-2$ unless $e=2k$ and $j=1$.  We have already shown that in the latter case,  the semigroup is not symmetric. 
\end{proof}

\begin{example}
\begin{enumerate}
    \item The Sally type semigroup $\S{8}{8}=\la 20,21,22,23,24,25,26,27,36,37,38,39 \ra$ is symmetric with Frobenius number $55=2e+k+j-1$ and genus $28=e+2k-j$. It satisfies $$e=20, \q k=8 \leq \frac{e}{2}=10, \q j=8=k.$$
     \item On the other hand, $\S{8}{7}=\la 20,21,22,23,24,25,26,35,36,37,38,39 \ra$ has Frobenius number $54$ and genus $29$ and is NOT symmetric. 
    \item $\S{10}{1}= \la 20,31,32,33,34,35,36,37,38,39\ra$ has Frobenius number $81=4e+1$ and genus $41=2e+1$. This is also symmetric and satisfies $$e=20, \q k= 10 = \frac{e}{2}, \q j=1.$$  
   
\end{enumerate}
    
\end{example}

\vspace{1cm}
    
\subsection{Minimal Free Resolutions and Betti Numbers}
We start by introducing some notation analogous to \Cref{not:A}.\\
Let $S_k:=\Se{k}{k}$ denote the symmetric semigroup minimally generated by $\langle e,\ldots e+k-1, e+2k, \ldots, 2e-1\rangle $. Let $R_k:=\Re{k}{k} = k[X_i, i \in [0,e-1] \bs [k,2k-1]]$ be the polynomial ring in $e-k$ variables. Denote the semigroup ring of $S_k$ by $$k[S_k]= k[t^{e+i}, i \in [0,e-1] \bs [k,2k-1]] \cong  \frac{R_k}{I_k},$$
where $I_k:= \Ie{k}{k}$ is its defining ideal.
Further, let us denote $X_{e+2i} = X_i^2; X_{e+2i+1} = X_iX_1$ for all $i\ge 0$.  
    \begin{theorem}\label{thm:kleqe/2ideal}
The ideal $I_k$ defining the semigroup ring $S_k$ is given by $I_2(A^e_k)+I_2(B^e_k)$, where 
$$A^e_k  = \begin{bmatrix}
        X_0 & X_1 &  \cdots  & X_{k-2} & X_{2k} & X_{2k+1} & \cdots  & X_{e-1} \\
        X_1  & X_2 &  \cdots & X_{k-1}  & X_{2k+1} & X_{2k+2} & \cdots  & X_0^2
        \end{bmatrix}$$
        and 
 $$   B^e_k =   \begin{bmatrix}
        {\bf X_{k-1}} &X_{2k} & X_{2k+1} & \cdots  & X_{e+k-4} &   X_{e+k-3} \\
       {\bf  X_{2k} }  & X_{3k+1} & X_{3k+2} & \cdots & X_{e+2(k-2)+1} & X_{e+2(k-1)}= X_{k-1}^2
        \end{bmatrix}$$  
    \end{theorem}
    \begin{proof} 
The ideal $J_h = I_2\left(\begin{bmatrix} y_0 &y_1\ldots y_{h-1}&y_{h}\\
y_1&y_2 \ldots y_{h}& y_0^2\\
\end {bmatrix}\right)$ in the polynomial ring $k[y_0,y_1,...,y_h]$ is prime of height $h$ since it is the ideal of the monomial curve parametrized by the arithmetic sequence $a, a+1, \ldots, a+h$. If we renumber $y_1, \ldots, y_h$ as $X_1, \ldots, X_{k-1}, X_{2k}, \dots, X_{e-1}$, with $h= e-k-2$, then  $I_2(A^e_k)+(X_{k-1}-X_{2k}) = J_{e-k-2}+(y_{k-1}-y_{2k})$.  Thus, height of $I_2(A^e_k)$ is $e-k-2$.  

So, $J:=I_2(A^e_k)+I_2(B^e_k)$ is a binomial ideal of height $e-k-1$ in $R_k$. Similar to the proof of \Cref{thm:remove2determ}, we will show that the $k$-vector space $V= R_k/(J,X_0)$ has dimension $e$. 

Indeed, $1, X_i, i \in [1,e-1]\bs [k,2k-1]$, are linearly independent in $V$ and are $e-k$ in number.  Now, $X_iX_{k-1}, 1\le i\le k-1 $ are also linearly independent. We will show that these along with $X_1X_{k-1}^2$ will be a basis for this vector space. 
 
Set $T = (1,X_i, i\in[1,e-1]\bs[k,2k-1], X_iX_{k-1}, 1\le i\le k-1, X_1X_{k-1}^2)$, ideal in $R_k$, and let $J' = J+(X_0)$ be the $m$-primary ideal of height $e-k$, where $m$ is the maximal ideal $(X_i, i\in[0,e-1]\bs[k,2k-1])$ in $R_k$.

From $I_2(A^e_k)$, for $i,j>0$,
\begin{equation}\label{eq:one}
    X_iX_j \in J' \text{ for } i+j \le k-1, \q \text{ and } \q  X_iX_j \in J' \text{ for } i\le k-1, j\ge 2k.
\end{equation}
Also, $X_iX_j-X_{k-1}X_{i+j-k+1}\in I_2(A^e_k)$ for $k \leq i+j \leq 2k-4$. Thus, 
\begin{equation*}
    X_iX_j \in J' \text{ for all } 1 \leq i,j \leq k-2. 
\end{equation*}

 Next, we claim that from $I_2(B^e_k)$, one gets
 \begin{equation}\label{eq:two}
     X_{2k}X_j \in J'+T, \; j\ge 2k.
 \end{equation}

This is because, if $3k+1 \le e-1$, then $X_{k-1}X_{3k+1}\in J'$ by \eqref{eq:one}, and hence, $X_{2k}^2 \in J'$.  Similarly, we will have $X_{2k}X_j \in J',$ for all $2k \leq j \le e-2$.  If $3k+1\ge e$, then $0 \leq 3k+1-e \leq k+1$. For $j\geq 0$, $X_{3k+1+j} =X_l^2$ or $X_1X_l$ for some $l\geq 0$. This is in $J'$ for all $l=3k+1+j-e \leq k-1$ (so $2k \leq e-2$). Again, we get $X_{2k}X_j \in J', 2k\le j\le e-2$.  Finally, $X_{2k}X_{e-1}-X_{k-1}(X_{e+k}-X_1X_{k-1})-X_1X_{k-1}^2 \in J'$.  Thus, $X_{2k}X_{e-1} \in J'+ T$. 

We can now use $I_2(A^e_k)$ and \eqref{eq:two} to conclude 
\begin{equation*}
    X_{i}X_j \in J' \text{ for all } i,j \geq 2k.
\end{equation*}

Next, all cubes $X_iX_jX_t \in J'$ if \begin{enumerate} \item $i,j,t \not\in \{k-1\}$ or \item $i=k-1$ and $j$ or $t$ is $\geq k$. \end{enumerate} Say $t=k-1, i,j\leq k-1$. Then $X_{k-1}X_iX_j -X_1^2X_{k-1} = X_{k-1}(X_iX_j-X_1X_{k-1}) \in J'+T$. So all cubes are either in the ideal $J'$ or are equal to $X_1X_{k-1}^2$ modulo $J'$. 
Since $X_{k-1}^3$ and $X_1^2$ are in the ideal $J'$, all other monomials of degree bigger than $3$ are in $J'$ as well. 

So, we have $\{1, X_i, i\ge 1, X_iX_{k-1}, i\le k-1, X_1X_{k-1}^2\}$ as a $k$-basis for the vector space $V$ and it has dimension $e$ as required. The conclusion now follows from $\Cref{thm:J=I}$.

 \end{proof}

Let $F = \sum_{i}R_kf_i$ be a free $R_k$-module of rank $e-k-1$ with basis $\{f_0,f_1,...,f_{k-2},f_{2k},f_{2k+1},...,f_{e-1}\}$ and $G=R_kg_1+R_kg_2$ another free module. Denote by $\mathbf E(k)$ be the Eagon Northcott resolution of $R_k / I_2(A^e_k)$ beginning with $\wedge^2(F)\to R_k$. $\mathbf{E^\ast}(k)$ denotes its dual.
    \begin{theorem}\label{thm:kleqe/2res}
The minimal $R_k$-resolution of the semigroup ring $S_k$ is given by the mapping cylinder of $\psi: \mathbf{E^*}(k)\to \mathbf E(k)$ induced by $\psi_0: \mathcal S_{e-k-3}G \to R_k$, which
    is defined by the $2\times 2$ minors of $B^e_k$ involving the first column. 
    \end{theorem}

\begin{proof}
Recall that $X_{e+2t} = (X_t)^2$ and $X_{e+2t-1}= X_{t-1}X_t, t\le k-1$. Thus, $X_iX_{j+1}-X_{i+1}X_j \in I_2(A^e_k)$ for all $0\le i\le j\le e+2k-2$. The $2 \times (e-k-1)$ matrix $A^e_k$ is of height $e-k-2$. Hence, $\mathbf E(k)$ is determined by the multiplication map $A^e_k: F\to G^*$.

\begin{tikzcd}
    \mathbf E(k): 0 \ar[r] & D_{e-k-3}G^\ast \ar[r,"\delta_{e-k-4}"] & \overset{e-k-2}{\bigwedge} F \otimes D_{e-k-4}G^\ast \ar[r, "\delta_{e-k-3}"] & \cdots \ar[r, "\delta_2"] & \overset{2}{\bigwedge}F \ar[r, "\delta_1"] & R_k \\
    \mathbf E^\ast(k): 0 \ar[r] & R_k \ar[r, "\delta_1^\ast"] \ar[u, dashed] & \overset{2}{\bigwedge} F^\ast \ar[r, "\delta_2^\ast"] \ar[u, dashed] & \cdots \ar[r, "\delta^\ast_{e-k-3}"] & F \otimes \mathcal S_{e-k-4}G \ar[r, "\delta^\ast_{e-k-2}"] \ar[u, dashed, "\psi_1"] & \mathcal S_{e-k-3}G \ar[u, "\psi"]
\end{tikzcd}
 
    We will define a map of complexes $\psi_{\cdot}: \mathbf E^*(k) \to \mathbf E(k)$ extending the map $\psi: S_{e-k-3}G\to R$ such that the composition
    $ F\otimes \mathcal S_{e-k-4} G \to \mathcal S_{e-k-3}G  \to R_k \to R_k/I_2(A^e_k)$ is zero. $$\psi(g_1^ig_2^{e-k-3-i}) =(-1)^i[X_{k-1}X_{3k+1+i}-X_{2k}X_{2k+i}], \q 0\le i\le e-k-3$$

    This is well-defined because $3k+1+i \le 3k+1+e+k-3 =e+2(k-1)$.  

    Now, $$\delta_{e-k-2}^* ( f_j  \otimes (g_1^{i}g_2^{e-k-4-i}))= (-1)^j [X_j g_1^{i+1}g_2^{e-k-4-i} +X_{j+1} g_1^{i}g_2^{e-k-3-i}].$$
Hence, the composition
    \begin{align*}
        & \psi \circ \delta_{e-k-2}^* ( f_j \otimes (g_1^{i}g_2^{e-k-4-i})) \\
        = & (-1)^j[(-1)^{i+1}X_j(X_{k-1}X_{3k+2+i}-X_{2k}X_{2k+i+1})+ (-1)^iX_{j+1} (X_{k-1}X_{3k+1+i}-X_{2k}X_{2k+i})] \\
        = & (-1)^{j+i+1}[ X_{k-1} ( X_jX_{3k+2+i}-X_{j+1} X_{3k+1+i}) + X_{2k}( X_jX_{2k+i+1}- X_{j+1}X_{2k})] \in I_2(A^e_k).
    \end{align*}

Thus, the mapping cylinder of $\psi$ is the resolution of $\frac{R_k}{I_2(A^e_k)+I_2(B^e_k)}= R_k/I_k$. 

Minimality of the resolution follows from the degree consideration of the maps. All the maps $\delta_i,$ $i\geq 2$, are of standard degree $1$ except in $X_0$.  The first map $\delta_1$ is of standard degree $\ge 2$ and is of degree $3$ whenever $X_0$ appears.  The map $\psi_1$ involves only $X_{k-1}$ and $X_{2k}$, and is always of degree $1$.  Thus, the non zero entries in the  map $\psi_2$ are also of degree $1$ or more.   It follows that the remaining $\psi_i$ are also minimal, except possibly the last one. However, the ideal $I_k$ is Cohen-Macaulay of co-dimension $1$ so the last map $\psi_{e-k-2}: R_k\to D_{e-k-3}G^\ast$ has to be minimal. 

Thus, the mapping cylinder of $\psi$ is the minimal resolution of the semigroup ring $k[S_k]$.
\end{proof}

\begin{corollary}
        Betti numbers of $S(k)$ are 
        $$\beta_t (S^e(k)) = t\displaystyle{{e-k-1 \choose t+1}+(e-k-1-t){e-k-1 \choose t-1}, \q  t\le e-k-2},$$
        and $\beta_{e-k-1} = 1$.
\end{corollary}

\vspace{1cm}

\section{Symmetric Sally Type Semigroups When $k> e/2$.}

In general, when $k>e/2$, the situation becomes much more chaotic. For instance, when $k$ is $e-1$ or $e-2$, $\S{k}{j}$ has embedding dimension $1$ or $2$, and so is trivially symmetric. When $k= e-3$, there are two possibilities for $\S{k}{j}$ to be symmetric: When $j=1$, it is symmetric if and only if $e$ is even. If $j=2$, it is symmetric if and only if $e\equiv 2$ mod $3$.  For $S_2(e-3) = \la e,e+1,2e-1 \ra$ is symmetric only if $(e+1, 2e-1) \neq 1$.  Hence $e = 2+3t$ for some $t \in \N$. Then $\la 3t+2, 3(t+1), 3(2t+1) \ra$ is symmetric.  Thus, we have symmetry if and only if $e\equiv 2$ mod $3$. Furthermore:

\begin{proposition}
    Assume $k>e/2$. If $e+j-2\geq 2k$ or $2j \geq k+1$, then $\S{k}{j}$ is never symmetric. 
\end{proposition}
\begin{proof}
    We have proved in Corollary \ref{cor:sym<e/2} that for $e+j-2\ge 2k$ or $2j\ge k+1$, symmetry does not occur unless $j=k$. Recall that the definition of Sally type semigroups implies $j+k \leq e-1$. However, since $k>e/2$, this implies $j \leq k-1$. So $\S{k}{j}$ is never symmetric.
\end{proof}

To make the problem more tractable in the remaining cases, i.e., when $\frac{e}{2} < k < e-3$, we first provide explicit formulae for the Frobenius number and genus of $\S{k}{j}$, each of which depends on all three parameters $e,k$ and $j$, and then characterize when $\S{k}{j}$ is symmetric. We divide this further into two cases.

\vspace{1cm}

\subsection{Symmetric Sally Type Semigroups when $k>e/2$ and $j \neq 1$.}

In this subsection, assume that 
\begin{equation}\label{eq:2k>e}
    \frac{e}{2} < k < e-3 \q \text{ and }  \q 2 \leq j \leq \frac{k}{2} \q \text{ and } \q e+j-2<2k
\end{equation}

\begin{theorem}\label{thm:table}
    The following table tells what the Frobenius number $F_i$ and genus $G_i$ can be for a given triple $e,k,j$ satisfying \eqref{eq:2k>e}.

\begin{figure}[h]
    \centering
    $$\begin{array}{|c|c|l|}
\hline
    i & \text{ Constraints } & \multicolumn{1}{c|}{F_i \text{ and } G_i} \\
    \hline
    \multirow{2}{*}{$3$} & \multirow{2}{*}{$e+j-1 \leq 2k < e+2j-2$} &  F_3=2e+2j+2k-1\\ 
    & & G_3=4k-2j+2 \\
    \hline 
     \multirow{2}{*}{$4$} & \multirow{2}{*}{$e+2j-2 \le 2k; 3k\le 2e+2j-4 $} & F_4=3e+2j+2k-1\\
    & & G_4=6k-4j-e+5 \\
    \hline 
    \vdots & \vdots & \multicolumn{1}{c|}{\vdots} \\
    \hline
    \multirow{2}{*}{$2n-1$} & (n-1)(e-2)+(n-1)j <nk  & F_{2n-1}=ne+nj+nk-1\\
    & nk \le (n-1)e + nj - (2n-1) & G_{2n-1} = n^2k-n(n-1)j -(n-1)^2e +(n-1)(2n-1)+e-1 \\
    \hline 
    \multirow{2}{*}{$2n$} & (n-1)e+nj-(2n-1)<nk;& F_{2n}=(n+1)e+nj+nk-1\\
    &  (n+1)k\le n e+nj-2n  & G_{2n} = n(n+1)k -n^2j -n(n-1)e+n(2n-1)+e-1 \\
    \hline 
\end{array}$$
    \caption{Frobenius Number and Genus for $\S{k}{j}$ given \eqref{eq:2k>e}.}
    \label{fig:table}
\end{figure}

\end{theorem}

\begin{figure}
    \centering
\begin{tikzpicture}
    \draw (0.75,8.25) node {{\small $g_0$}};
    \draw (0,8) -- (14.5,8);
    \draw (0,7.9) -- (0,8.1);
    \draw (0,7.6) node {$0$};
    \draw (1.5,7.5) -- (1.5,8.5) node[anchor=south] {$e$};
    \draw (1.5,7.5) -- (1.6,7.5);
    \draw (1.5,8.5) -- (1.6,8.5);
    \draw (2.5,7.25) node {{\small $j$}};
    \draw (3.6,7.5) -- (3.6,8.5) node[anchor=south] {$e+j-1$};
    \draw (3.5,7.5) -- (3.6,7.5);
    \draw (3.5,8.5) -- (3.6,8.5);
    \draw (4.5,8.25) node {{\small $g_1$}};
    \draw (5.5,7.5) -- (5.5,8.5) node[anchor=south] {$e+j+k$};
    \draw (5.5,7.5) -- (5.6,7.5);
    \draw (5.5,8.5) -- (5.6,8.5);
    \draw (6.75,7.25) node {{\small $e+j-k-1$}};
    \draw (8.1,7.5) -- (8.1,8.5) node[anchor=south] {$2e+2j-2$};
    \draw (8,7.5) -- (8.1,7.5);
    \draw (8,8.5) -- (8.1,8.5);
    \draw (9.25,8.25) node {{\small $g_2$}};
    \draw (10.5,7.5) -- (10.5,8.5) node[anchor=south] {$2e+j+k$};
    \draw (10.5,7.5) -- (10.6,7.5);
    \draw (10.5,8.5) -- (10.6,8.5);
    \draw (12,7.25) node {{\small $e+2j-k-2$}};
    \draw (13.6,7.5) -- (13.6,8.5) node[anchor=south] {$3e+3j-3$};
    \draw (13.5,7.5) -- (13.6,7.5);
    \draw (13.5,8.5) -- (13.6,8.5);
    \draw (14.25,8.25) node {{\small $g_3$}};
     \draw (1.75,5.25) node {{\small $g_3$}};
    \draw (1,5) -- (15.5,5);
    \draw (3.5,4.5) -- (3.5,5.5) node[anchor=south] {$2e+2j+2k$};
    \draw (3.5,4.5) -- (3.6,4.5);
    \draw (3.5,5.5) -- (3.6,5.5);
    \draw (5,4.25) node {{\small $2e+2j-2k-3$}};
    \draw (6.6,4.5) -- (6.6,5.5) node[anchor=south] {$4e+4j-4$};
    \draw (6.5,4.5) -- (6.6,4.5);
    \draw (6.5,5.5) -- (6.6,5.5);
    \draw (8,5.25) node {{\small $g_4$}};
    \draw (9.5,4.5) -- (9.5,5.5) node[anchor=south] {$3e+2j+2k$};
    \draw (9.5,4.5) -- (9.6,4.5);
    \draw (9.5,5.5) -- (9.6,5.5);
    \draw (11,4.25) node {{\small $2e+3j-2k-4$}};
    \draw (12.6,4.5) -- (12.6,5.5) node[anchor=south] {$5e+5j-5$};
    \draw (12.5,4.5) -- (12.6,4.5);
    \draw (12.5,5.5) -- (12.6,5.5);
    \draw (14,5.25) node {{\small $g_5$}};
   \draw (1.75,2.25) node {{\small $g_5$}};
    \draw (0.5,2) -- (14.25,2);
    \draw[dashed] (14.25,2) -- (15.25,2);
    \draw (3,1.5) -- (3,2.5) node[anchor=south] {$3e+3j+3k$};
    \draw (3,1.5) -- (3.1,1.5);
    \draw (3,2.5) -- (3.1,2.5);
    \draw (4.5,1.25) node {{\small $3e+3j-3k-5$}};
    \draw (6.1,1.5) -- (6.1,2.5) node[anchor=south] {$6e+6j-6$};
    \draw (6,1.5) -- (6.1,1.5);
    \draw (6,2.5) -- (6.1,2.5);
    \draw (7.5,2.25) node {{\small $g_6$}};
    \draw (9,1.5) -- (9,2.5) node[anchor=south] {$4e+3j+3k$};
    \draw (9,1.5) -- (9.1,1.5);
    \draw (9,2.5) -- (9.1,2.5);
    \draw (10.5,1.25) node {{\small $3e+4j-3k-6$}};
    \draw (12.1,1.5) -- (12.1,2.5) node[anchor=south] {$7e+7j-7$};
    \draw (12,1.5) -- (12.1,1.5);
    \draw (12,2.5) -- (12.1,2.5);
    \draw (13.5,2.25) node {{\small $g_7$}};
    \draw (1,-0.75) node {{\small $g_{2n-1}$}};
    \draw[dashed] (0,-1) -- (0.5,-1);
    \draw (0.5,-1) -- (15,-1);
    \draw[dashed] (15,-1) -- (15.5,-1);
    \draw (2,-1.5) -- (2,-0.5) node[anchor=south] {$ne+nj+nk$};
    \draw (2,-1.5) -- (2.1,-1.5);
    \draw (2,-0.5) -- (2.1,-0.5);
    \draw (3.5,-1.75) node {{\small $ne+nj-nk-(2n-1)$}};
    \draw (5.1,-1.5) -- (5.1,-0.5) node[anchor=south] {$2ne+2nj-2n$};
    \draw (5,-1.5) -- (5.1,-1.5);
    \draw (5,-0.5) -- (5.1,-0.5);
    \draw (6.7,-0.75) node {{\small $g_{2n}$}};
    \draw (8.5,-1.5) -- (8.5,-0.5) node[anchor=south] {$(n+1)e+nj+nk$};
    \draw (8.5,-1.5) -- (8.6,-1.5);
    \draw (8.5,-0.5) -- (8.6,-0.5);
    \draw (10.75,-1.75) node {{\small $ne+(n+1)j-nk-2n$}};
    \draw (13.1,-1.5) -- (13.1,-0.5) node[anchor=south] {$(2n+1)e+(2n+1)j-(2n+1)$};
    \draw (13,-1.5) -- (13.1,-1.5);
    \draw (13,-0.5) -- (13.1,-0.5);
    \draw (14.25,-0.75) node {{\small $g_{2n+1}$}};
\end{tikzpicture}
    \caption{Gaps and Nongaps in $\S{k}{j}$}
    \label{fig:numberline}
\end{figure}
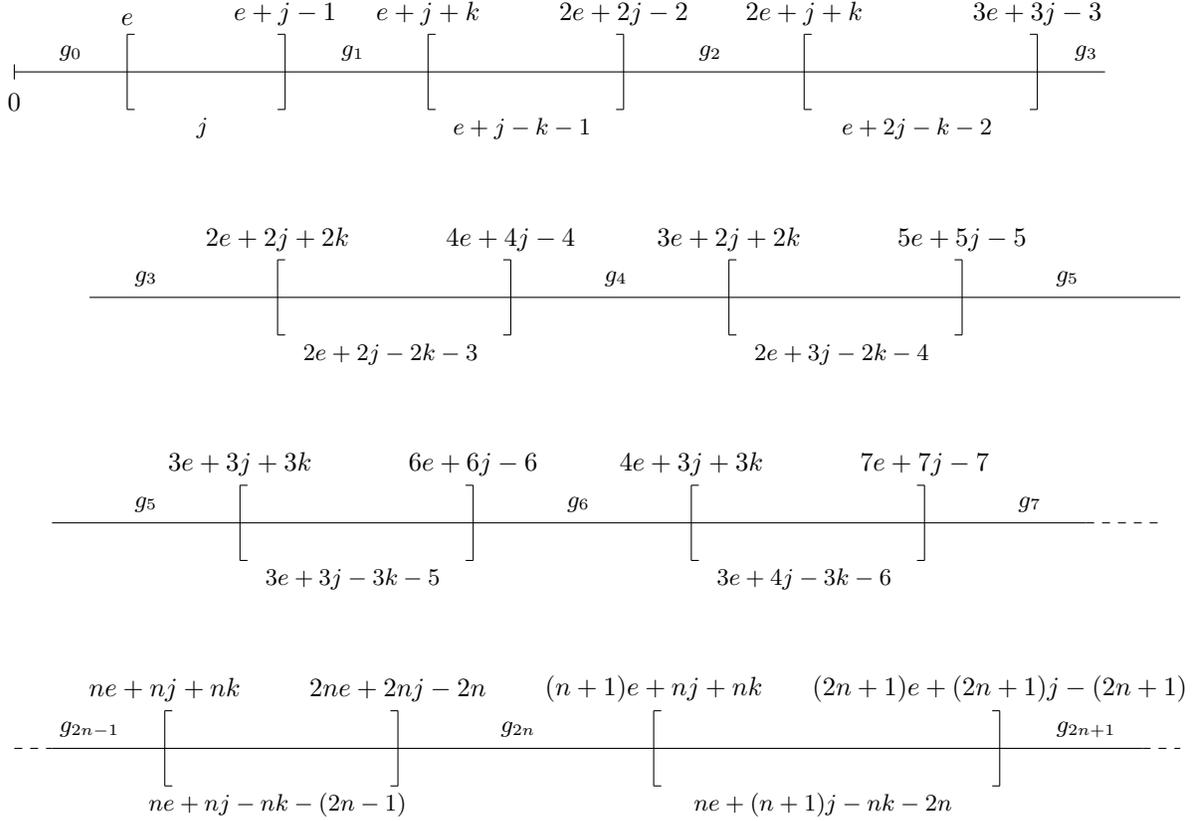

\begin{proof}
     Consider the number line plotted with all the integer points. \Cref{fig:numberline} illustrates the intervals contained in $\S{k}{j}$ and the possible gaps $g_i$.  For instance, $g_0$ is the cardinality of the set $[1,e-1]$, which are all gaps of $\S{k}{j}$. $g_1$ is the cardinality of $[e+j,e+j+k-1]$. The intervals $[e,e+j-1]$ and $[e+j+k,2e+2j-2]$, with cardinality $j$ and $e+j-k-1$ respectively, are always contained in $\S{k}{j}$.  Further, $g_2=2e+j+k-1-2e-2j+2 = k-j+1 >0$ and $g_3 = 2e+2j+2k-1-3e-3j+3=2k-e-j+2>0$. 
     
Then we claim that $G_i = \sum_{t=0}^i = g_i$ and $F_i$ are as above.  
From \Cref{fig:numberline}, we see that for $n\ge 1$, 
$$g_{2n-1} = nk-(n-1)e-(n-1)j -2(n-1); \; \; g_{2n} = nk-(n-1)e-nj+(2n-1).$$

So $F= F_1, \; G= G_1= g_0+g_1 = e-1+k$ if and only if $k \le j-1$. Next, $F = F_2,\; G= G_2= G_1+g_2 = e+2k-j$, if and only if $j-1 <k ;\;  2k\le e+2j-2$. Thus, we have the result for $n=1$.  

We now proceed by induction on $n$. Assume we have the result for all $k\le n-1$. Necessarily $F> F_{2n-2}, G\ge G_{2n-2}$. This implies $g_{2n-1}>0$, i.e., $nk > (n-1)e+(n-1)j -2(n-1)$.  Now, $F = F_{2n-1}$ if and only if  \begin{itemize}
         \item[(i)] $(n-1)e+nj-nk-(2n-1) \geq 0$, or 
         \item[(ii)] $g_{2n}=0 \implies nk \le (n-1)e+nj-(2n-1)$.
     \end{itemize}
     Thus, $F= F_{2n-1}$ and $G= G_{2n-2}+g_{2n-1}= G_{2n-1}$ if $ (n-1)e+(n-1)j -2(n-1) <nk \le  (n-1)e+nj -(2n-1)$.  Now let $nk >  (n-1)e+nj -2(n-1)$, so that $F \ge F_{2n-1}$ and $g_{2n} \neq 0$.  So, $F= F_{2n}$ if and only if \begin{itemize}
         \item[(i)] $(n-1)e+(n+1)j-nk-2n \geq 0$, or 
         \item[(ii)] $g_{2n+1}=0 \implies (n+1)k \le ne+nj-2n$.
     \end{itemize}

The inequality in (i) implies the one in (ii), when we factor in $j+k\le e-1$.   So, we only need the second. 
     Thus, $F= F_{2n}$ and $G= G_{2n-1}+g_{2n}=G_{2n}$ if and only if $(n-1)e+nj-(2n-1)<nk; (n+1)k\le n e+nj-2n$. This finishes the proof. 
\end{proof}

\begin{example}\label{eg:2k=e+4j-4}
    Take $e=20, k=12, j=2$. Then $S^{20}_{12}(2)= \la 20,21,34,35,36,37,38,39 \ra$ satisfies
    \begin{itemize}
        \item[] $10=\frac{e}{2}< k \leq e-4=16$,
        \item[] $2 \leq j \leq \frac{k}{2} = 6$,
        \item[] $e+3j-3=23 \leq 2k=24 < \frac{3e}{2}+2j-2 = 34$.
    \end{itemize}
     By GAP (\cite{NumericalSgps}), its Frobenius number and genus are $F_4=87 \text{ and } G_4=25.$
    This fits the table above since $2k=24=e+4j-4$ and $3k \le 2e+3j-4$. 
\end{example}

We have seen in theorem \ref{thm:sym<e/2} that for $2k\le e$, $\S{k}{j}$ is symmetric if and only if $j=k$. As noted earlier, things are not so simple when $k> e/2$. In fact, we may not always have a symmetric Sally type semigroup. More precisely:

\begin{proposition}
    There is no symmetric Sally type semigroup at the $2n-1$ level. That is, if $$(n-1)(e-2)+(n-1)j < nk \leq (n-1)e+nj-(2n-1),$$ so that $\S{k}{j}$ has Frobenius number and genus $F_{2n-1}$ and $G_{2n-1}$, respectively, then $\S{k}{j}$ is not symmetric.  
\end{proposition}
\begin{proof}
Suppose there is an $n$ such that $(n-1)(e-2)+(n-1)j <nk \le (n-1)e + nj - (2n-1)$. This implies
\begin{equation}\label{eq:nosym2n-1jnot1}
    (n-1)j < nk-(n-1)(e-2) < nj
\end{equation}

so $j\neq 1$. In this case, $F_{2n-1} = ne+nj+nk-1; G_{2n-1} = n^2k-n(n-1)j -(n-1)^2e +(n-1)(2n-1)+e-1$. Then $\S{k}{j}$ will be symmetric if and only if $F_{2n-1}+1 = 2G_{2n-1}$, i.e.,
\begin{align*}
   &  ne+nj+nk= 2(n^2k-n(n-1)j-(n-1)^2e +(n-1)(2n-1)+e-1) \\
   \Leftrightarrow \; & e(2n^2-4n+2+n-2)= (2n^2-n)(k-j)+ 2(2n^2-3n) \\
   \Leftrightarrow \; & (2n-1)(k-j) = (2n-3)(e-2).
\end{align*}
So, $e-1= (2n-1)t$ for some $t \in \N$ and $ k = (2n-3)t+j$. Putting this in \eqref{eq:nosym2n-1jnot1}, we get  $(n-1)(2n-1)t-n(2n-3)t<j \implies t<j$. But $j+k \leq e-2 \implies 2 \leq  e-k =2t+2-j \implies j \leq t$, a contradiction.
\end{proof}

From \Cref{thm:table}, there is one remaining possibility: $$(n-1)e+nj-(2n-1)<nk; (n+1)k\le n e+nj-2n.$$ Rewrite this to get 
\begin{equation}\label{eq:2ncaseineq}
    \frac{n-1}{n} \leq  \frac{k-j}{e-2} <\frac{n}{n+1}.
\end{equation} 
We can now characterize exactly when these $\S{k}{j}$ are symmetric.
\begin{theorem}\label{thm:symk>e/2jnot1}
Let $2k>e$.  For any given $e, j < k$, $0< \frac{k-j}{e-2} <1$ since $e-k \ge 2$ and $2k>e$. So, there exists $n$ depending on $k, j, e$ such that $\frac{n-1}{n}\leq \frac{k-j}{e-2} <\frac{n}{n+1}$.  Then,  $\S{k}{j}$  is symmetric if and only if   $n(k-j) = (n-1)(e-2)$.  
\end{theorem}
\begin{proof}
First suppose there is an $n$ satisfying \eqref{eq:2ncaseineq}. So $F_{2n} = (n+1)e+nj+nk-1;  G_{2n} = n(n+1)k -n^2j -n(n-1)e+n(2n-1)+e-1$. Hence, we have symmetry if and only if 
\begin{align*}
   & (n+1)e+nj+nk= 2 (n(n+1)k -n^2j -n(n-1)e+n(2n-1)+e-1) \\
   \Leftrightarrow \; & (2n-1)n(k-j)+(2n-1)2(n-1) =  (n+1+2n^2-2n-2)e \\
    \Leftrightarrow \; & n(k-j) =(n-1)(e-2)
\end{align*}

This means $e-2= nt, k-j = nt-t$ for some $t \in \N$. So, $k= e-2-t+j = e-2-(t-j)$.  In particular, $t-j\ge 2$. Now we check that such an $n$ does exist: $(n-1)e-(2n-1) = (n-1)(nt+2)-(2n-1) =(n-1)nt-1 =n(k-j)-1 < n(k-j$, thus, satisfying the first inequality at the $2n$ level in \Cref{thm:table}. For the second inequality,  $(n+1)k = (n+1)((n-1)t+j) =  n^2t +nj + j-t \le  n(nt+2)+nj-2n ne+nj -2n$ if $j\le t$. Since $t-j \ge 0$, this is certainly possible.

This finishes the proof.
\end{proof}

\begin{example} \begin{enumerate}
    \item Observe that $S_{12}^{20}(2)$ from Example \ref{eg:2k=e+4j-4} is at level $2n$ for $n=2$  and is not symmetric.
    \item Take $e=26, j=4,k=16$, so $S_{16}^{26}(4)=\la 26,27,28,29,50,51 \ra$ is at level $2n$ for $n=3$. Here, $t=8$, and $S_{16}^{26}(4)$ is symmetric with Frobenius number $F_{6}=175$ and genus $G_6=88$.
\end{enumerate}
    
\end{example}

Thus, when $2k>e$, we may not always have a symmetric Sally type semigroup. When we do, that is, when $F = (n+1)e+nj+nk-1$ and $n(k-j)= (n-1)(e-2)$, we can write the structure of the ideal and the resolution in a similar manner to the case $2k\le e$.   

First, some more notation analogous to \Cref{not:A} to avoid clutter in the following proofs: From \Cref{thm:symk>e/2jnot1}, there exists $n_k$ satisfying \eqref{eq:2ncaseineq}. Denote the semigroup $\S{k}{j}$ at the $2n_k$ level by $S(n_k)$. Let $R(n_k):=\Re{k}{j}$ be the polynomial ring, and $k[S{n_k)}] \cong  \frac{R(n_k)}{I(n_k)}$ the corresponding semigroup ring with defining ideal $I(n_k):= \Ie{k}{j}$. Recall once more that $X_{e+2i} = X_i^2; X_{e+2i+1} = X_iX_1$ for all $i\ge 0$. 
 
\begin{theorem} Let $n_k$ be such that $ \frac {n}{n-1} <\frac {k-j}{e-2} < \frac{n}{n+1}$.  Then the defining ideal of $k[S(n_k)]$ is given by $I(n_k) = I_2(A^e_k(1))+I_2(B^e(n_k)$ where $$A^e(n_k) = \begin{bmatrix} X_0& \ldots &X_{j-2} &   X_{j+k} & \cdots  & X_{e-1} \\ X_{1}& \ldots &X_{j-1}& X_{j+k+1}& \cdots  & X_0^2 \end{bmatrix}$$ and  $$   B^e(n_k) =   \begin{bmatrix} X_{j-1}^{2n-1} & X_{k+j} & X_{k+j+1} & \cdots & X_{e+j-3} \\ X_{j+k}^{n} & X_{k+2j+1} & X_{k+2j+2} & \cdots & X_{e+2j-2}\\ \end{bmatrix}$$  

The minimal resolution of $k[S(n_k)]$ is the mapping cylinder of $\psi: \mathbf E^*(n_k)\to \mathbf E(n_k)$ where $\mathbf E(n_k)$ is the Eagon Northcott resolution of $R(n_k)/I_2(A^e(n_k))$, $\mathbf E^*(n_k)$ is its dual, and $\psi$ is the complex map induced by $\psi_0: R(n_k)^{e+k-2} \to R(n_k)$ given by the $2\times 2$ minors of $B^e(n_k)$ involving the first column. Thus, the Betti numbers are $$\beta_t   = t\displaystyle {e-k-1 \choose t+1}+(e-k-1-t)\displaystyle{e-k-1 \choose t-1}, \q t\le e-k-2.$$ and $\beta_{e-k-1} =1$. 
\end{theorem}

\begin{proof} We follow the proof idea of the case $k \leq e/2$. First, we show $J:=I_2(A(n_k))+I_2(B(n_k)) \subseteq I(n_k)$. We only need to check the minors $\begin{bmatrix}
    X^{2n-1}_{j-1} & X_{k+j} \\ X_{j+k}^n & X_{k+2j+1} 
\end{bmatrix}$ and $\begin{bmatrix}
    X^{2n-1}_{j-1} & X_{k+j} \\ X_{j+k}^n & X_{k+2j+1} 
\end{bmatrix}$. The rest of them are in $I(n_k)$ because of homogenous degree considerations. For the first minor, $(e+j-1)(2n-1)+(e+k+2j+1) - (e+j+k)(n+1) = (e-2)(n-1)+ en-j + e+k+2j - (k-j)n-en -e-j-k= 0$. Further, for the second minor, $(e+j-1)(2n+1) - (e+j+k)n -( 2e+j-3) = e(n-1) +j(2n+1-n-1)-(2n+1) -nk +3= (n-1)(e-2)-n(k-j)=0$. Thus, we have $J \subset I(S)$.  
 
Next, we will check that the dimension of the $k$-vector space $R(n_k)/(J+(X_0))$ is $e$ by seeing that
\begin{multline*}
    T:=\{1\} \cup \{X_i \mid i \in [1,2e-1]\bs[j,j+k]\} \cup \{X_iX_{j-1}^u \mid 1\le i\le j-1,1\le u\le 2n-1\}\cup \\ \{X_{i} X_{j+k}^u \mid j+k\le i\le e-1, 1\le u\le n-1\} \cup \{ X_1X_{j-1}^{2n} \text { or } X_{j-2}X_{j-1}^{2n}\}
\end{multline*}
is a basis. This can be done as in the proof of \Cref{thm:kleqe/2ideal}. For instance, from the matrix $A^e(n_k)$, we have
\begin{equation*}
    X_aX_b \in J+(X_0) \; \forall \;  a\le j-1, b\ge j+k \text{ and } \; \forall \;  a+b\le j-1; \q
    X_{j+k+a}X_{b}\in J+(X_0), \q a+b\ge e.
\end{equation*}
We therefore need the first four sets in a basis.  Finally, if $j <4$, then we will need $X_{j-2}X_{j-1}^{2n}$ and if $j>4,$ then we will need $X_1X_{j-1}^{2n}$. Since all the other monomials are in $J+(X_0)+T$, $T$ is indeed a basis, and we have
\begin{multline*}
    \dim_k \left(\frac{R(n_k)}{J+(X_0)}\right)=  e-k+(2n-1)(j-1)+(e-j-k)(n-1)+1 = ne+nj-nk -2n+2 \\ = n(e-2)-n(k-j)+2 = e-2+2= e 
\end{multline*}

Thus, by Gastinger's criterion (\Cref{thm:J=I}), we get $I(n_k)= I_2(A^e(n_k))+I_2(B^e(n_k))$.

One can similarly check that the computation of the resolution and Betti numbers is analogous to those for $\S{k}{k}$ (see \Cref{thm:kleqe/2res}).
 
\end{proof}

\vspace{1cm}

\subsection{Symmetric Sally Type Semigroups when $k>e/2$ and $j=1$.}
Let $\S{k}{1} = \la e,e+k+1, \ldots, 2e-1 \ra$. Assume $k \le e-3$ so that the embedding dimension is at least $3$. There exists $n\ge 1$ such that $\frac{n-2}{n-1} <\frac {k}{e-1} \le \frac{n-1}{n}$. Let us denote the semigroup $\S{k}{1}$ with this $n$ by $S(n_{k})$. 
\begin{theorem}\label{thm:symk>e/2j=1}
    Let $S(n_k) =\la e,e+k+1, \dots, 2e-1 \ra$ be the semigroup $\S{k}{1}$ with $\frac{n-2}{n-1} <\frac {k}{e-1} \le \frac{n-1}{n}$. Then its frobenius number and genus are $$F = ne+(n-1)k+(n-2); \q G= n(n-1)k-(n^2-3n+1)(e-1).$$   
Further, $S(n_k)$ is symmetric if and only if $(n-1)k = (n-2)e-(n-3)$.
    \end{theorem}

    \begin{proof}
        This proof is by induction on $n$. When $n=2$, that is, when $2k\le e-1$, we already know that $S(n_k)$ is symmetric if and only if $k=1$.  
        We can also see from the proof of \ref{thm:sym<e/2} that in this case, $F= 2e+j+k-1= 2e+k$ and $G= e-1+2k$. 
        
        Assume the result is true for some $n-1$. So $G \ge (n-1)(n-2)k-((n-1)^2-3(n-1)+1)(e-1)$. Now suppose 
        $$(n-1)k > (n-2)(e-1) \text { and } nk\le (n-1)(e-1).$$ 
        Further, $[(n-1)e+(n-2)k+n-2, (2n-3)e] \subset S$ has length $e+(n-3)e-(n-2)k -(n-3)$.  This length is $< e$ because
        $\frac{k}{e-1} > \frac {n-2}{n-1} > \frac {n-3}{n-2}$.  However, the next set of gaps is non-empty as $(n-1)e+(n-1)k+(n-1) -((2n-3)e+1)= (n-1)k-(n-2)(e-1) >0$.  So, $F \ge (n-1)e+(n-1)k+(n-2)$ and $G \ge (n-1)(n-2)k-((n-1)^2-3(n-1)+1)(e-1)+(n-1)k-(n-2)(e-1)$.  
        We note that the length of the next interval in the semigroup $[(n-1)e+(n-1)j+(n-1), (2n-2)e]$ also falls short of $e$ as $(n-1)k > (n-2)(e-1)$. The next set pf gaps between $2n-2$ and $ne+(n-1)j+(n-1)k$ is of length 
        $(n-1)k-(n-2)(e-1)$ which is positive for the same reason. The next interval $[ne+(n-1)k+(n-1), (2n-1)e]\subseteq S(n_k)$ yet  again falls short but then the gaps from $(2n-1)e$ to $ne+nk+n$  disappear for $ne+nk+n \le (2n-1)e+1 $  when $nk\le (n-1) (e-1)$.
        So, we get $F = ne+(n-1)k+(n-2)$ and 
        $G = (n-1)(n-2)k-((n-1)^2-3(n-1)+1)(e-1)+2(n-1)k-2(n-2)(e-1) = n(n-1)k-(n^2-3n+1)(e-1)$ as desired.  
        Thus we have the formulae for $F$ and $G$. 

        Now, $S(n_k)$ is symmetric if and only if $F+1= 2G$.  This happens if $$ ne+(n-1)k+(n-1) =  2n(n-1)k-2(n^2-3n+1)(e-1)$$

That is, $$ (2n^2-5n+2)e =  (2n-1)(n-1)k+ (2n^2-6n+2-n+1).$$  But $(2n^2-5n+2= (2n -1)(n-2)$ and $2n^2-7n+3= (2n-1)(n-3)$.
Hence $S(n_k)$ is symmetric if and only if $(n-1)k = (n-2)e-(n-3)$ which is just possible as $(n-2)e-(n-3)= (n-2)(e-1)+1$. 
    \end{proof}
\begin{example}
Here we give some examples for these symmetric $S(n_k)$.  Let $n=7$.  If the embedding dimension is $4$, then we get that $\la 20, 37,38,39 \ra$ is the only symmetric Sally type semigroup of the form $S(n_k)$. If the embedding dimension is $5$, then the only symmetric one is $S(n_k)=\la 26, 48, 49, 50, 51 \ra$. Indeed, if we fix $n$, there is only one symmetric Sally type semigroup in every embedding dimension. 
\end{example}

We will complete this analysis for $j=1$, by giving the structure of the ideal as well as the minimal resolution. 
    
    \begin{theorem}
The ideal defining the semigroup ring $k[S(n_k)]$ is given by $I(n_k)=I_2(A^e(n_k))+I_2(B^e(n_k))$ where 
$$A^e(n_k) = \begin{bmatrix}
      X_{k+1} & X_{j+k} & \cdots  & X_{e-1} \\
     X_{K+2} & X_{j+k+1}& \cdots  & X_0^2
        \end{bmatrix}$$
        and 
 $$   B^e(n_k) =   \begin{bmatrix}
 X_{0}^{2n-3} & X_{k+1} & X_{k+2} & \cdots & X_{e-2} \\
 X_{k+1}^{n-1} & X_{k+3} & X_{k+4} & \cdots & X_{0}^2=X_{e}\\
        \end{bmatrix}$$  
        
        The minimal resolution of $k[S(n_k)]$ is the mapping cylinder of $\psi: \mathbf E^*(n_k)\to \mathbf E(n_k)$ where $\mathbf E(n_k)$ is the Eagon-Northocott resolution of $R(n_k)/I_2(A^e(n_k))$ above and $\psi$ is the complex map induced by $\psi_0: (R(n_k))^{e+k-2} \to R(n_k)$ given by the $2\times 2$ minors of $B^e(n_k)$ involving the first column. 
        The Betti numbers of $k[S(n_k)]$ are $$\beta_t   = t\displaystyle {e+k-1 \choose t+1}+(e-k-1-t)\displaystyle{e-k-1 \choose t-1}.$$
    \end{theorem}
\begin{proof}
The proof is once again very similar to the case of $\S{k}{k}$. We will simply check that the minors  $\begin{bmatrix}
    X^{2n-3}_{0} & X_{k+1} \\ X_{k+1}^{n-1} & X_{k+3} 
\end{bmatrix}$ and $\begin{bmatrix}
    X^{2n-3}_{0} & X_{e-2} \\ X_{k+1}^n & X_{0}^2 
\end{bmatrix}$ of $I_2(B^e(n_k))$ are in $I(n_k)$. 
This we do by looking at the homogenous degrees again. For the first minor: $(2n-3)e+e+k+3-n(e+k+1)=(n-2)e-(n-3)-k(n-1)=0$ by \Cref{thm:symk>e/2j=1} as $S(n_k)$ is symmetric. The last minor is $X_0^{2n-1}- X_{k+1}^{n-1}X_{e-2}$.  Now,
$(2n-1)e- (e+k+1)(n-1) -(2e-2) = e(n-2)- k(n-1)-n+1+2=0$. 

Next, for $J:=I_2(A^e(n_k))+I_2(B^e(n_k)$,  $\dim_k\left(\frac{R(n_k)}{(J+X_0)}\right)$ is $e$ because it has a $k$-basis, 
\begin{multline*}
    \{1\} \cup \{X_i \mid i\in[k+1,e-1]\} \cup \{X_{k+p}X_{k+1}^q \mid  1\le p\le e-k-1, 1\le 1\le n-2\} \cup \left\{ X_{e-2}X_{k+1}^{n-1}\right\}
\end{multline*}

The number of elements in this basis is exactly $1+e-k-1+ (e-k-1)(n-2)+ 1= e$.
\end{proof}

\section{Conclusion}
To summarise, for a Sally type numerical semigroup $\S{k}{j}$, the bottom line of the flowchart gives the necessary and sufficient criterion for symmetry.

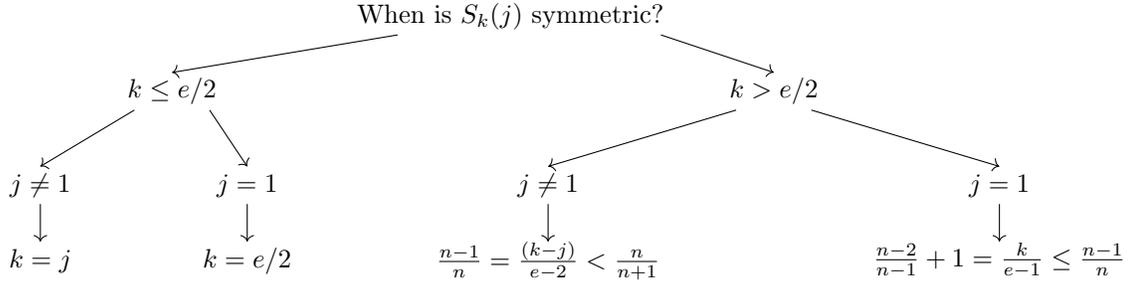
\begin{figure}[h]
    \centering
\begin{tikzpicture}
    \draw (6.5,10.25) node {When is $\S{k}{j}$ symmetric?};
    \draw[->] (5,10) -- (2,9.5);
    \draw[->] (8.5,10) -- (10,9.5);
    \draw (2,9.25) node {$k \leq e/2$};
    \draw (10,9.25) node {$k>e/2$};
    \draw[->] (1.5,9) -- (0.25,8.25);
    \draw[->] (2.5,9) -- (3,8.25);
    \draw[->] (9.5,9) -- (7,8.25);
    \draw[->] (10.5,9) -- (13,8.25);
    \draw (0.25,8) node {$j \neq 1$};
    \draw (3,8) node {$j=1$};
    \draw (7,8) node {$j \neq 1$};
    \draw (13,8) node {$j=1$};
    \draw[->] (0.25,7.75) -- (0.25,7.25);
    \draw[->] (3,7.75) -- (3,7.25);
    \draw[->] (7,7.75) -- (7,7.25);
    \draw[->] (13,7.75) -- (13,7.25);
    \draw (0.25,7) node {$k=j$};
    \draw (3,7) node {$k=e/2$};
    \draw (7,7) node {$\frac{n-1}{n}= \frac{(k-j)}{e-2} <\frac{n}{n+1}$};
    \draw (13,7) node {$\frac{n-2}{n-1}+1=\frac{k}{e-1}\le \frac{n-1}{n}$};
\end{tikzpicture}
\end{figure}
We end with a few examples of symmetry or non-symmetry of a Sally type semigroup when the elements removed from $[e,2e-1]$ are not consecutive.

\begin{example}
    $$\begin{array}{|c|c|}
    \hline
        \text{Sally Type Semigroup} &  \text{ Symmetric? } \\
        \hline
         \la 10,15,17,18,19 \ra & \text{ Yes } \\
         \la 10,15,16,17,19 \ra & \text{ No } \\
         \hline
         \la 10,11,16,18,19 \ra & \text{ No } \\
         \la 10,11,16,17,19 \ra & \text{ No } \\
         \hline
         \la 10,13,16,18,19 \ra & \text{ No } \\
         \la 10,13,16,17,19 \ra & \text{ Yes } \\
         \hline
    \end{array}$$
\end{example}

\printbibliography

\end{document}